\documentclass[12pt]{amsart}
\usepackage{amsmath}
\usepackage{amssymb}

\newtheorem{theorem}{Theorem}[section]
\newtheorem{lemma}[theorem]{Lemma}
\newtheorem{proposition}[theorem]{Proposition}

 \theoremstyle{definition}
\newtheorem{definition}[theorem]{Definition}
\newtheorem{assumption}[theorem]{Assumption}

\theoremstyle{remark}
\newtheorem{remark}[theorem]{Remark}

\numberwithin{equation}{section}

\begin{document}

\title{The $L_p$-dual Christoffel-Minkowski problem for the case $p\geq q$}

\author{Xiaojuan Chen}
\address{Faculty of Mathematics and Statistics, Hubei Key Laboratory of Applied Mathematics, Hubei University,  Wuhan 430062, P.R. China}
\email{201911110410741@stu.hubu.edu.cn}

\author{Qiang Tu$^{\ast}$}
\address{Faculty of Mathematics and Statistics, Hubei Key Laboratory of Applied Mathematics, Hubei University,  Wuhan 430062, P.R. China}
\email{qiangtu@hubu.edu.cn}

\author{Ni Xiang}
\address{Faculty of Mathematics and Statistics, Hubei Key Laboratory of Applied Mathematics, Hubei University,  Wuhan 430062, P.R. China}
\email{nixiang@hubu.edu.cn}

\thanks{This research was supported by funds from Hubei Provincial Department of Education
Key Projects D20181003 and the National Natural Science Foundation of China No. 11971157, 12101206.}
\thanks{$\ast$ Corresponding author}


\date{}

\begin{abstract}
In this paper, we consider a class of Hessian equations associated to the $L_p$-dual Christoffel-Minkowski problem for the case $p\geq q$. By combining the tools of constant rank theorem, the a priori estimates and the continuity method, we obtain the existence and uniqueness for strictly spherical convex solutions to the $L_p$-dual Christoffel-Minkowski problem.
\end{abstract}


\makeatletter
\@namedef{subjclassname@2020}{\textup{2020} Mathematics Subject Classification}
\makeatother

\maketitle
\vskip4ex

\section{Introduction}

The classical Brunn-Minkowski theory is the classical core of the geometry
of convex bodies. The Minkowski sum, the mixed volumes, curvature and area measures are fundamental concepts. The introduction of dual curvature measures and their variational formulas by Huang, Lutwak, Yang, and Zhang \cite{HLYZ16} has significantly expanded the classical theory, leading to the development of the dual Brunn-Minkowski theory.

This paper concerns the $L_p$-dual Christoffel-Minkowski problem, which is an extension of the classical Minkowski problem. It involves finding a convex body whose dual curvature measures match a given measure, under the $L_p$ norm.
When the given measure has a density,  the specific equation can be  reduced to the following Hessian type equation
 \begin{equation}\label{G-eq}
\sigma_k(u_{ij}+u\delta_{ij})=u^{p-1}(u^2+|\nabla u|^2)^{\frac{k+1-q}{2}}\varphi(x),\quad on~ \mathbb{S}^{n},
\end{equation}
where $\sigma_k$ is the $k$-th elementary symmetric polynomial, $u_{ij}$ is the second order covariant derivative of $u$ with respect to orthonormal frames on $\mathbb{S}^n$, $\delta_{ij}$ is the standard Kronecker
symbol and $\varphi$ is a positive smooth function on $\mathbb{S}^{n}$.

When $p=1,q=k+1$, equation \eqref{G-eq} corresponds to the classical Christoffel-Minkowski problem
\begin{equation}\label{sol}
\sigma_k(u_{ij}+u\delta_{ij})=f(x), \quad on~ \mathbb{S}^{n},
\end{equation}
which has attracted much attentions. In the case $k=1$, equation \eqref{sol} is the classical Christoffel problem, the early treatments were given in Christoffel \cite{Ch65} and others, the final solutions were obtained in Firey \cite{Fi67,Fi68} and Berg \cite{Be69}. In the case $k=n$, equation \eqref{sol} corresponds to the classical Minkowski problem, which has been settled by the works of Minkowski \cite{M87}, Alexandrov \cite{Al37}, Lewy \cite{LH38}, Nirenberg \cite{NL53}, Cheng-Yau \cite{CY76} and Pogorelov \cite{PA78}. In the intermediate case $1<k<n$, equation \eqref{sol} is precisely the celebrated Christoffel-Minkowski problem, which has been widely investigated in
 Guan-Ma \cite{GM03}, Guan-Lin-Ma \cite{GM06} and Guan-Ma-Zhou \cite{GM10}.



The Christoffel-Minkowski problem related to $p$-sums, which can be called the $L_p$-Christoffel-Minkowski problem
\begin{equation}\label{mat}
  \sigma_k(u_{ij}+u\delta_{ij})=u^{p-1}\varphi(x),\quad on~ \mathbb{S}^{n}.
\end{equation}
For $k=n$, equation \eqref{mat} corresponds to the $L_p$-Minkowski problem, which was introduced by Lutwak \cite{LE93} and then has been extensively studied.
One can refer to the logarithmic Minkowski problem in \cite{BLYZ13, CLZ19, Z14}, the centroaffine Minkowski problem in \cite{JLW18,JLZ16,L18,L19,LW13,Z15}, the others in \cite{BBC19,HLYZ10,HH12,JL19}. For the general $k$, we refer the readers to Hu-Ma-Shen \cite{HMS04} for $p\geq k+1$, and Guan-Xia \cite{GX18} for $1<p<k+1$, respectively. One can consult  \cite{IV19, SY20}  for more works.

The $L_p$-dual Christoffel-Minkowski problem \eqref{G-eq} contains all the aforementioned Christoffel-Minkowski problems,
however, as far as we know, there is relatively few research for the general $k$ in equation \eqref{G-eq}.
Recently, Li-Ju-Liu \cite{LJL22} and Ding-Li \cite{DL22} have  obtained the existence and uniqueness for solutions of the $L_p$-dual Christoffel-Minkowski problem by the flow methods, respectively.




A solution $u$ of equation \eqref{G-eq} is called admissible if $(u_{ij}+u\delta_{ij})\in \Gamma_k$ and $u$ is (strictly) spherical convex if $(u_{ij}+u\delta_{ij})\geq 0(>0)$.
 To address the existence of convex bodies for the $ L_p $-dual Christoffel-Minkowski problem, especially when the given measure has a density, we need to focus on the solvability of strictly spherical convex solutions for the relevant equation \eqref{G-eq}.
In order to obtain the solvability and ensure the maintenance of convexity,  constant rank theorem plays a crucial role.
Before elaborating the relevant conclusions, we propose the following assumptions.
\begin{assumption}\label{cdt-091}
Let $\varphi(x)$ be a positive smooth function satisfying one of the following conditions:

(1) if $p\geq1, q\leq k+1$, $\left(\varphi^{-\frac{1}{k+p-1}}\right)_{ii}+\varphi^{-\frac{1}{k+p-1}}\geq0$;


(2) if $p\geq 1, k+1<q<2k+p$, $\left(\varphi^{-\frac{1}{k+p-1}}\right)_{ii}+\frac{2k+p-q}{k+p-1}\varphi^{-\frac{1}{k+p-1}}\geq0$.
\end{assumption}
Then the main theorem is as follows.
\begin{theorem}\label{exist-main}
Let $1\leq k \leq n$ and $\varphi$ be a positive smooth function satisfying Assumption \ref{cdt-091}.

(1) If $p>q$, then there exists a unique positive strictly spherical convex solution $u$ of equation \eqref{G-eq}.

(2) If $p=q>1$, then there exists a unique positive constant $\gamma$ such that
\begin{equation}\label{Keq}
  \sigma_k(u_{ij}+u\delta_{ij})=u^{p-1}(u^2+|\nabla u|^2)^{\frac{k+1-q}{2}}\gamma\varphi(x),\quad on~ \mathbb{S}^{n}
\end{equation}
has a unique positive strictly spherical convex solution $u$ up to a dilation.
\end{theorem}
\begin{remark}
Although the condition $p\geq q$ is  not necessary for constant rank theorem, it is essential to derive $C^0, C^1$ estimates, the existence and uniqueness.
\end{remark}
\begin{remark}
  Theorem \ref{exist-main} contains the results of the classical $L_p$-Christoffel-Minkowski problem. Specifically, when $q=k+1$, equation \eqref{G-eq} becomes   the $L_p$-Christoffel-Minkowski problem equation \eqref{mat}. The condition $\left(\varphi^{-\frac{1}{k+p-1}}\right)_{ii}+\varphi^{-\frac{1}{k+p-1}}\geq0$ in Assumption \ref{cdt-091} is sufficient for the existence of equation \eqref{mat} when $p\geq k+1$.
  \end{remark}

\begin{remark}
  For the general $L_p$-Christoffel-Minkowski problem
  $$\sigma_k(u_{ij}+u\delta_{ij})=\varphi(x)g(u),$$
  the existence still holds with the following assumptions:
  $$(\log \varphi)_{ii}\leq 0,\quad gg_{zz}\leq g_z^2,\quad g_zz\geq (p-1)g, \quad \forall ~p>k+1,$$
  $$\lim_{z\rightarrow 0^+} \frac{g(z)}{z^k}=0,\quad \lim_{z\rightarrow + \infty} \frac{g(z)}{z^k}=+\infty.$$
  Some nonhomogeneous cases are also included in our discussion. For instance,
  $$g(z)=z^{p-1}\ln(1+z^s),\  \ \  s>0;$$ $$g(z)=z^{p-1}e^{\sum_ic_iz^{\gamma_i}};$$  $$g(z)=z^{p-1}(\sum_ic_iz^{\gamma_i}+c),  \gamma_i\in(0,1], \ c_i,\ c>0.$$
 \end{remark}

The rest of the paper is organized as follows. In Section 2, we start
with some preliminaries. In  Section 3, we prove constant rank theorem for equation \eqref{G-eq} in maintaining the convexity of solutions. The a priori estimates, the existence and uniqueness in Theorem \ref{exist-main} in the case $p>q$ and $p=q>1$ are established in Section 4 and Section 5, respectively.

\section{Preliminaries}

\subsection{Basic properties of convex hypersurface}
Let $\mathcal{M}$ be a smooth, closed, uniformly convex hypersurface in $\mathbb{R}^{n+1}$. Assume that $\mathcal{M}$ is parametrized by the inverse Gauss map
$$X:\mathbb{S}^n\rightarrow \mathcal{M}.$$
The support function $u:\mathbb{S}^n\rightarrow\mathbb{R}$ of $\mathcal{M}$ is defined by
\begin{equation*}
  u(x)=\sup\{\langle x,y\rangle:y\in\mathcal{M}\}.
\end{equation*}
The supremum is attained at a point $y$ such that $x$ is the outer normal of $\mathcal{M}$ at $X$. It is easy to check that
$$X=u(x)x+\nabla u(x),$$
where $\nabla$ is the covariant derivative with respect to the standard metric $\sigma_{ij}$ of the sphere $\mathbb{S}^n$. Hence
\begin{equation}\label{rho}
\rho=|X|=\sqrt{u^2+|\nabla u|^2}.
\end{equation}
The second fundamental form of $\mathcal{M}$ is given by
\begin{equation}\label{hij}
  h_{ij}=u_{ij}+u\sigma_{ij},
\end{equation}
where $u_{ij}=\nabla_{ij}u$ denotes the second order covariant derivative of $u$ with respect to the spherical metric $\sigma_{ij}$. By Weingarten formula
\begin{equation}\label{sij}
  \sigma_{ij}=\langle\nabla_ix, \nabla_jx\rangle=h_{ik}g^{kl}h_{jl},
\end{equation}
where $g_{ij}$ is the metric of $\mathcal{M}$ and $g^{ij}$ is its inverse. It follows from \eqref{hij} and \eqref{sij} that the principle radii of curvature of $\mathcal{M}$, under a smooth local orthonormal frame on $\mathbb{S}^n$, are the eigenvalues of the matrix
$$b_{ij}=u_{ij}+u\delta_{ij}.$$
In particular, the Gauss curvature is given by
$$K=\frac{1}{\det(u_{ij}+u\delta_{ij})}.$$

\subsection{$k$-th elementary symmetric functions}
Let $\lambda=(\lambda_1,\cdots,\lambda_n)\in\mathbb{R}^n$, then we recall
 definitions of elementary symmetric function for $1\leq k\leq n$
\begin{equation*}
\sigma_k(\lambda)= \sum _{1 \le i_1 < i_2 <\cdots<i_k\leq
n}\lambda_{i_1}\lambda_{i_2}\cdots\lambda_{i_k}.
\end{equation*}

\begin{definition}
Let $1\leq k\leq n$ and $\Gamma_k$ be a cone in $\mathbb{R}^n$ determined by
$$\Gamma_k  = \{ \lambda  \in \mathbb{R}^n :\sigma _i (\lambda ) >
0,~\forall~ 1 \le i \le k\}.$$
\end{definition}

Denote $\sigma_{k-1}(\lambda|i)=\frac{\partial
\sigma_k}{\partial \lambda_i}$ and
$\sigma_{k-2}(\lambda|ij)=\frac{\partial^2 \sigma_k}{\partial
\lambda_i\partial \lambda_j}$, then we list some properties of
$\sigma_k$ which will be used later.

\begin{proposition}\label{sigma}
Let $\lambda=(\lambda_1,\cdots,\lambda_n)\in\mathbb{R}^n$ and $1\leq
k\leq n$. Then we have

(1) $\Gamma_1\supset \Gamma_2\supset \cdot\cdot\cdot\supset
\Gamma_n$;

(2) $\sigma_{k-1}(\lambda|i)>0$ for $\lambda \in \Gamma_k$ and
$1\leq i\leq n$;

(3) $\sigma_k(\lambda)=\sigma_k(\lambda|i)
+\lambda_i\sigma_{k-1}(\lambda|i)$ for $1\leq i\leq n$;

(4)
$\sum_{i=1}^{n}\frac{\partial[\frac{\sigma_{k}}{\sigma_{l}}]^{\frac{1}{k-l}}}
{\partial \lambda_i}\geq [\frac{C^k_n}{C^l_n}]^{\frac{1}{k-l}}$ for
$\lambda \in \Gamma_{k}$ and $0\leq l<k$;

(5) $\Big[\frac{\sigma_k}{\sigma_l}\Big]^{\frac{1}{k-l}}$ are
concave in $\Gamma_k$ for $0\leq l<k$;

(6) If $\lambda_1\geq \lambda_2\geq \cdot\cdot\cdot\geq \lambda_n$,
then $\sigma_{k-1}(\lambda|1)\leq \sigma_{k-1}(\lambda|2)\leq
\cdot\cdot\cdot\leq \sigma_{k-1}(\lambda|n)$ for $\lambda \in
\Gamma_k$;

(7)
$\sum_{i=1}^{n}\sigma_{k-1}(\lambda|i)=(n-k+1)\sigma_{k-1}(\lambda)$.
\end{proposition}

\begin{proof}
All the properties are well known. For example, see Chapter XV in
\cite{Li96} or \cite{Hui99} for proofs of (1), (2), (3),  (6) and
(7); see Lemma 2.2.19 in \cite{Ger06} for the proof of (4); see
\cite{CNS85} and \cite{Li96} for the proof of (5).
\end{proof}

\begin{proposition}
Let $W=W_{ij}$ be an $n\times n$ symmetric matric, $\lambda(W)=(\lambda_1,\lambda_2,\cdots,\lambda_n)$ be the eigenvalues of the symmetric matric $W$. Suppose that $W=W_{ij}$ is diagonal and $\lambda_i=W_{ii}$, then we have
$$\frac{\partial\lambda_i}{\partial W_{ij}}=\delta_{ij},$$
\begin{eqnarray*}
\frac{\partial^2\lambda_i}{\partial W_{ij}\partial W_{pq}}=
\begin{cases}
\frac{1}{\lambda_i-\lambda_p},~& i=q, j=p, i\neq p,\\
0,~~~& \mbox{otherwise}.
\end{cases}
\end{eqnarray*}
\end{proposition}

\begin{proposition}
Suppose $W=W_{ij}$ is diagonal and $m(1\leq m\leq n)$ is a positive integer, then
\begin{eqnarray*}
\frac{\partial \sigma_m(W)}{\partial W_{ij}}=
  \begin{cases}
  \sigma_{m-1}(W|i), ~&i=j,\\
  0, ~~~&otherwise,
  \end{cases}
\end{eqnarray*}
\begin{eqnarray*}
\frac{\partial^2\sigma_m(W)}{\partial W_{ij}\partial W_{pq}}=
\begin{cases}
\sigma_{m-2}(W|ip), ~~~&i=j,p=q,i\neq p,\\
-\sigma_{m-2}(W|ip), ~~~&i=q,j=p,i\neq j,\\
0,~~~&\mbox{otherwise}.
\end{cases}
\end{eqnarray*}
\end{proposition}

The generalized Newton-MacLaurin inequality is as follows.
\begin{proposition}\label{NM}
For $\lambda \in \Gamma_m$ and $m > l \geq 0$, $ r > s \geq 0$, $m
\geq r$, $l \geq s$, we have
\begin{align}
\Bigg[\frac{{\sigma _m (\lambda )}/{C_n^m }}{{\sigma _l (\lambda
)}/{C_n^l }}\Bigg]^{\frac{1}{m-l}} \le \Bigg[\frac{{\sigma _r
(\lambda )}/{C_n^r }}{{\sigma _s (\lambda )}/{C_n^s
}}\Bigg]^{\frac{1}{r-s}}. \notag
\end{align}
\end{proposition}
\begin{proof}
See \cite{S05}.
\end{proof}

\section{Constant rank theorem for equation \eqref{G-eq}}
To ensure the maintenance of the convexity in solutions when employing  the continuity method, we can utilize a specific version of
 the following constant rank theorem.
\begin{theorem}\label{crt}
Suppose $u:\mathbb{S}^n \rightarrow \mathbb{R}$ is a solution of equation \eqref{G-eq} such that $(u_{ij}+u\delta_{ij})$ is semi-positive definite on $\mathbb{S}^n$, $\varphi$ is a positive smooth function satisfying one of the following conditions:

(1) if $p\geq1, q\leq k+1$, $\left(\varphi^{-\frac{1}{k+p-1}}\right)_{ii}+\varphi^{-\frac{1}{k+p-1}}\geq0$;


(2) if $p\geq1, k+1<q<2k+p$, $\left(\varphi^{-\frac{1}{k+p-1}}\right)_{ii}+\frac{2k+p-q}{k+p-1}\varphi^{-\frac{1}{k+p-1}}\geq0$.\\
Then $(u_{ij}+u\delta_{ij})$ is positive definite on $\mathbb{S}^n$.
\end{theorem}

\begin{proof}
Suppose $W=(u_{ij}+u\delta_{ij})$ attains its minimal rank $l$ at some point $x_0\in\mathbb{S}^n$, we have $\sigma_l(W)(x_0)>0$ and $\sigma_{l+1}(W)(x_0)=0$. Then there exists a small open neighborhood $\mathcal{O}$ of $x_0$ and a small positive constant $c_0$ such that $\sigma_l(W)(x_0)\geq c_0>0$. We may assume that $l\leq n-1$, otherwise we are done.

For convenience, we denote $\lambda=(\lambda_1,\cdots,\lambda_n)$ and $\lambda_i$ are eigenvalues of $W$. For each $x\in\mathcal{O}$, we can rotate coordinate such that $W=(b_{ij})$ is diagonal and $b_{11}\leq b_{22}\leq\cdots\leq b_{nn}$ at $x$.
Consider the test function
$$\phi(x)=\sigma_{l+1}(W)+\frac{\sigma_{l+2}}{\sigma_{l+1}}(W).$$
Following the notations in \cite{Gu02}, we say that $h(y)\lesssim k(y)$ provided there exist positive constants $c_1$ and $c_2$ such that
$$(h-k)(y)\leq\left(c_1|\nabla\phi|+c_2\phi\right)(y),$$
and we write $h(y)\sim k(y)$ if $h(y)\lesssim k(y)$ and $k(y)\lesssim h(y)$.

Let $D=\{1,\cdots,n-l\}$ and $G=\{n-l+1,\cdots,n\}$ be the sets of indices for eigenvalues $\lambda_i$. Let $\Lambda_D=(\lambda_1, \cdots, \lambda_{n-l})$ be the ``bad" eigenvalues of $W$ and $\Lambda_G=(\lambda_{n-l+1}, \cdots, \lambda_n)$ be the ``good" eigenvalues of $W$, for convenience, we also write $D=\Lambda_D, G=\Lambda_G$ if there is no confusion. Hence we get
\begin{equation}\label{equ1}
  0\sim\phi(x)\sim\sigma_{l+1}(W)\sim\sigma_l(G)\sum_{i\in D}b_{ii}\sim\sum_{i\in D}b_{ii}.
\end{equation}
For the convenience of calculation, equation \eqref{G-eq} can be expressed as
\begin{equation}\label{eq1}
 -\sigma_k^{-\frac{1}{k}}(u_{ij}+u\delta_{ij})=u^t(u^2+|\nabla u|^2)^{\frac{s}{2}}\widetilde{\varphi}(x):=\widetilde{f},
\end{equation}
where $s=-\frac{k+1-q}{k}$, $t=-\frac{p-1}{k}$ and $\widetilde{\varphi}=-\varphi^{-\frac{1}{k}}$.
Denote $F=-\sigma_k^{-\frac{1}{k}}$, then
 $$F^{ij}=\frac{\partial F}{\partial b_{ij}},\quad
F^{ij,rs}=\frac{\partial^2F}{\partial b_{ij}\partial b_{rs}},\quad\widetilde{f}_i=\frac{\partial\widetilde{f}}{\partial x_i},\quad\widetilde{f}_{ij}=\frac{\partial^2\widetilde{f}}{\partial x_i\partial x_j}.$$
 Differentiating equation \eqref{eq1} twice, we obtain
$$F^{\alpha\beta}b_{\alpha\beta i}=\widetilde{f}_i, \quad F^{\alpha\beta}b_{\alpha\beta ii}+F^{\alpha\beta,rs}b_{\alpha\beta i}b_{rsi}=\widetilde{f}_{ii}.$$
Following the idea of \cite{BG09} (can also see \cite{CX22}), we get
\begin{eqnarray}\label{eqnar1}
 \nonumber F^{\alpha\beta}\phi_{\alpha\beta} &=&O(\phi+\sum_{i,j\in D}|\nabla b_{ij}|)-\frac{1}{\sigma_1(D)}\sum_{\alpha}\sum_{i\neq j\in D}F^{\alpha\alpha}b_{ij\alpha}^2 \\
 \nonumber &&-\frac{1}{\sigma_1(D)^3}\sum_{\alpha}\sum_{i\in D}F^{\alpha\alpha}(b_{ii\alpha}\sigma_1(D)-b_{ii}\sum_{j\in D}b_{jj\alpha})^2\\
 \nonumber &&-2\sum_{i\in D}\left(\sigma_l(G)+\frac{\sigma_1(D|i)^2-\sigma_2(D|i)}{\sigma_1(D)^2}\right)\sum_{\alpha, j\in G}F^{\alpha\alpha}\frac{b_{j\alpha i}^2}{b_{jj}}\\
  &&+\sum_{i\in D}\left(\sigma_l(G)+\frac{\sigma_1(D|i)^2-\sigma_2(D|i)}{\sigma_1(D)^2}\right)\sum_{\alpha}F^{\alpha\alpha}b_{ii\alpha\alpha}.
\end{eqnarray}
For any $i\in D$, we have
\begin{eqnarray}\label{eqnar2}
 \nonumber F^{\alpha\alpha}b_{ii\alpha\alpha} &=& F^{\alpha\alpha}(b_{\alpha\alpha ii}+b_{ii}-b_{\alpha\alpha}) \\
 \nonumber &=& -F^{\alpha\beta,rs}b_{\alpha\beta i}b_{rsi}+\widetilde{f}_{ii}+\widetilde{f}+O(\phi)\\
  &=& -\sum_{\alpha,\beta,r,s\in G}F^{\alpha\beta,rs}b_{\alpha\beta i}b_{rsi}+\widetilde{f}_{ii}+\widetilde{f}+O(\phi+\sum_{i,j\in D}|\nabla b_{ij}|).
\end{eqnarray}
We claim that
\begin{equation}\label{claim}
  \sum_{i\in D}(\widetilde{f}_{ii}+\widetilde{f})\leq O(\phi+\sum_{i,j\in D}|\nabla b_{ij}|).
\end{equation}
If the claim is true, then combining with \eqref{eqnar1}-\eqref{claim}, we derive
\begin{eqnarray}\label{eqnar3}
  \nonumber F^{\alpha\beta}\phi_{\alpha\beta} &\leq& O(\phi+\sum_{i,j\in D}|\nabla b_{ij}|)-\frac{1}{\sigma_1(D)}\sum_{\alpha}\sum_{i\neq j\in D}F^{\alpha\alpha}b_{ij\alpha}^2\\
 \nonumber&&-\frac{1}{\sigma_1(D)^3}\sum_{\alpha}\sum_{i\in D}F^{\alpha\alpha}(b_{ii\alpha}\sigma_1(D)-b_{ii}\sum_{j\in D}b_{jj\alpha})^2\\
  &&-\sum_{i\in D}\left(\sigma_l(G)+\frac{\sigma_1(D|i)^2-\sigma_2(D|i)}{\sigma_1(D)^2}\right)\nonumber\\
  &&\cdot\left(2\sum_{\alpha,j\in G}F^{\alpha\alpha}\frac{b_{j\alpha i}^2}{b_{jj}}+\sum_{\alpha,\beta,r,s\in G}F^{\alpha\beta,rs}b_{\alpha\beta i}b_{rsi}\right).
\end{eqnarray}
In fact, $F(W^{-1})=-\sigma_k^{-\frac{1}{k}}(W^{-1})=-\left(\frac{\sigma_n(W)}{\sigma_{n-k}(W)}\right)^{\frac{1}{k}}$ is convex with respect to $W$, i.e., $F(W)$ is ``inverse convex'' with respect to $W$. It is equivalent to
$$F^{\alpha\beta,rs}X_{\alpha\beta}X_{rs}+2\frac{F^{\alpha r}}{b_{\beta s}}X_{\alpha\beta}X_{rs}\geq 0, \quad \forall~(X_{\alpha\beta})\in \mbox{Sym}(n).$$
Taking $X_{\alpha\beta}=-b_{\alpha\beta i}$ for $\alpha, \beta\in G$ and otherwise $X_{\alpha\beta}=0$, then we obtain
\begin{equation}\label{inco}
2\sum_{\alpha,j\in G}F^{\alpha\alpha}\frac{b_{j\alpha i}^2}{b_{jj}}+\sum_{\alpha,\beta,r,s\in G}F^{\alpha\beta,rs}
b_{\alpha\beta i}b_{rsi}\geq 0.
\end{equation}
Hence by \eqref{eqnar3}, \eqref{inco} and the idea in \cite{BG09} (can also see \cite{CX22}), we get
\begin{eqnarray*}
  F^{\alpha\beta}\phi_{\alpha\beta} &\leq&O(\phi+\sum_{i,j\in D}|\nabla b_{ij}|)-\frac{1}{\sigma_1(D)}\sum_{\alpha}\sum_{i\neq j\in D}F^{\alpha\alpha}b_{ij\alpha}^2\\
 &&-\frac{1}{\sigma_1(D)^3}\sum_{\alpha}\sum_{i\in D}F^{\alpha\alpha}(b_{ii\alpha}\sigma_1(D)-b_{ii}\sum_{j\in D}b_{jj\alpha})^2\\
  &\leq& O(\phi+|\nabla\phi|)\\
  &\lesssim& 0,
\end{eqnarray*}
for any $x\in \mathcal{O}$.
Therefore by the strong minimum principle, $\phi\equiv0$ in $\mathcal{O}$, thus $\{x:\phi(x)=0\}$ is an open and closed set. So $\phi\equiv0$, i.e., $W=(u_{ij}+u\delta_{ij})$ is of constant rank on $\mathbb{S}^n$. Then the Minkowski integral formula (see \cite{GM03}) implies $W$ is of full rank, the proof is finished. So the rest is to prove the following claim:
$$\sum_{i\in D}(\widetilde{f}_{ii}+\widetilde{f})\leq O(\phi+\sum_{i,j\in D}|\nabla b_{ij}|).$$

\emph{Proof of Claim:} Since $b_{ii}=u_{ii}+u=O(\phi)$, then we derive
\begin{eqnarray*}
  \sum_{i\in D}(\widetilde{f}_{ii}+\widetilde{f}) &=& \sum_{i\in D}(\widetilde{f}_{x_ix_i}+2\widetilde{f}_{x_iz}u_i+\widetilde{f}_{zz}u_i^2)+\sum_k\sum_{i\in D}\widetilde{f}_{p_k}u_{iik}+\sum_{i\in D}\widetilde{f}\\
  && +\sum_{i\in D}u_{ii}(2\widetilde{f}_{x_ip_i}+\widetilde{f}_z+2\widetilde{f}_{zp_i}u_i+\widetilde{f}_{p_ip_i}u_{ii})\\
  &=&O(\phi+\sum_{i,j\in D}|\nabla b_{ij}|)+\sum_{i\in D}(\widetilde{f}_{x_ix_i}+2\widetilde{f}_{x_iz}u_i+\widetilde{f}_{zz}u_i^2-2\widetilde{f}_{x_ip_i}u\\
  &&-\widetilde{f}_zu-2\widetilde{f}_{zp_i}uu_i+\widetilde{f}_{p_ip_i}u^2-\sum_k\widetilde{f}_{p_k}u_k+\widetilde{f}).
\end{eqnarray*}
Denote $I:=\sum_{i\in D}(\widetilde{f}_{x_ix_i}+2\widetilde{f}_{x_iz}u_i+\widetilde{f}_{zz}u_i^2-2\widetilde{f}_{x_ip_i}u
-\widetilde{f}_zu-2\widetilde{f}_{zp_i}uu_i+\widetilde{f}_{p_ip_i}u^2-\sum_k\widetilde{f}_{p_k}u_k+\widetilde{f})$, we only need to derive $I\leq0$, then the claim is proved.

By direct calculation, we get
\begin{eqnarray}\label{pku}
 \nonumber I&=&u^t(u^2+|\nabla u|^2)^{\frac{s}{2}}\sum_{i\in D}\bigg(\widetilde{\varphi}_{ii}+\frac{2tu_i\widetilde{\varphi}_i}{u}+\frac{t(t-1)u_i^2\widetilde{\varphi}}{u^2}\\
  &&+\frac{su_i^2\widetilde{\varphi}}
  {u^2+|\nabla u|^2}-\frac{s\sum_ku_k^2\widetilde{\varphi}}{u^2+|\nabla u|^2}+(1-t)\widetilde{\varphi}\bigg),
\end{eqnarray}
where $s,t,\widetilde{\varphi}$ are defined in \eqref{eq1}.
Then we have
\begin{equation}\label{CS}
\frac{2tu_i\widetilde{\varphi}_i}{u}+\frac{t(t-1)u_i^2\widetilde{\varphi}}{u^2}\leq\frac{t\widetilde{\varphi}_i^2}{(1-t)\widetilde{\varphi}},
\end{equation}
with $t\leq0$ or $t>1$.\\
Next we continue the proof with four cases.

\emph{Case 1:} $t\leq0, s\leq0$, i.e., $p\geq1, q\leq k+1$.\\
When $s\leq0$, we get
\begin{eqnarray}\label{ca1}
  \nonumber\sum_{i\in D}\left(\frac{su_i^2\widetilde{\varphi}}
  {u^2+|\nabla u|^2}-\frac{s\sum_ku_k^2\widetilde{\varphi}}{u^2+|\nabla u|^2}\right) &=& \frac{s\widetilde{\varphi}\sum_{i\in D}u_i^2}{u^2+|\nabla u|^2}-\frac{s(n-l)\widetilde{\varphi}\sum_ku_k^2}{u^2+|\nabla u|^2} \\
  \nonumber &\leq& \frac{s\widetilde{\varphi}\sum_{i\in D}u_i^2}{u^2+|\nabla u|^2}-\frac{s\widetilde{\varphi}\sum_ku_k^2}{u^2+|\nabla u|^2}\\
   \nonumber&=&-\frac{s\widetilde{\varphi}\sum_{i\in G}u_i^2}{u^2+|\nabla u|^2}\\
   &\leq&0.
\end{eqnarray}
By \eqref{pku}-\eqref{ca1} and the assumption (1) on $\varphi$, we derive
\begin{eqnarray*}
  &&I \leq u^t(u^2+|\nabla u|^2)^{\frac{s}{2}}\sum_{i\in D}\left(\widetilde{\varphi}_{ii}+\frac{t\widetilde{\varphi}_i^2}{(1-t)\widetilde{\varphi}}+(1-t)\widetilde{\varphi}\right)\\
  &=&u^{-\frac{p-1}{k}}(u^2+|\nabla u|^2)^{-\frac{k+1-q}{2k}}\sum_{i\in D}\left(-(\varphi^{-\frac{1}{k}})_{ii}+\frac{p-1}{p+k-1}\frac{(\varphi^{-\frac{1}{k}})_i^2}{\varphi^{-\frac{1}{k}}}-\frac{p+k-1}{k}\varphi^{-\frac{1}{k}}\right)\\
  &=&-\frac{p+k-1}{k}u^{-\frac{p-1}{k}}(u^2+|\nabla u|^2)^{-\frac{k+1-q}{2k}}\varphi^{-\frac{p-1}{k(k+p-1)}}\sum_{i\in D}\left((\varphi^{-\frac{1}{k+p-1}})_{ii}+\varphi^{-\frac{1}{k+p-1}}\right)\\
  &\leq& 0.
\end{eqnarray*}

\emph{Case 2:} $t\leq0, s>0$, i.e., $p\geq1, q> k+1$.\\
When $s> 0$, we get
\begin{equation}\label{sge}
  \sum_{i\in D}\left(\frac{su_i^2\widetilde{\varphi}}
  {u^2+|\nabla u|^2}-\frac{s\sum_ku_k^2\widetilde{\varphi}}{u^2+|\nabla u|^2}\right)\leq-s\sum_{i\in D}\widetilde{\varphi}.
\end{equation}
By \eqref{pku}, \eqref{CS}, \eqref{sge} and the assumption (2) on $\varphi$, we have
\begin{eqnarray*}
  &&I\leq u^t(u^2+|\nabla u|^2)^{\frac{s}{2}}\sum_{i\in D}\left(\widetilde{\varphi}_{ii}+\frac{t\widetilde{\varphi}_i^2}{(1-t)\widetilde{\varphi}}+(1-t-s)\widetilde{\varphi}\right) \\
  &=& -u^{-\frac{p-1}{k}}(u^2+|\nabla u|^2)^{-\frac{k+1-q}{2k}}\sum_{i\in D}\left((\varphi^{-\frac{1}{k}})_{ii}-\frac{p-1}{p+k-1}\frac{(\varphi^{-\frac{1}{k}})_i^2}{\varphi^{-\frac{1}{k}}}
  +\frac{2k+p-q}{k}\varphi^{-\frac{1}{k}}\right)\\
  &=&-\frac{p+k-1}{k}u^{-\frac{p-1}{k}}(u^2+|\nabla u|^2)^{-\frac{k+1-q}{2k}}\varphi^{-\frac{p-1}{k(k+p-1)}}\\
  &&\cdot\sum_{i\in D}\left((\varphi^{-\frac{1}{k+p-1}})_{ii}+\frac{2k+p-q}{k+p-1}\varphi^{-\frac{1}{k+p-1}}\right)\\
  &\leq&0.
\end{eqnarray*}
At the maximum point of $\varphi^{-\frac{1}{k+p-1}}$, the condition $\left(\varphi^{-\frac{1}{k+p-1}}\right)_{ii}+\frac{2k+p-q}{k+p-1}\varphi^{-\frac{1}{k+p-1}}\geq0$ is contradict with $\varphi>0, p\geq 1, q\geq2k+p$, hence we need $p\geq1, k+1<q<2k+p$.

\emph{Case 3:} $t>1, s\leq0$, i.e., $p<1-k, q\leq k+1$.\\
Similar to Case 1, we also have
$$I\leq-\frac{p+k-1}{k}u^{-\frac{p-1}{k}}(u^2+|\nabla u|^2)^{-\frac{k+1-q}{2k}}\varphi^{-\frac{p-1}{k(k+p-1)}}\sum_{i\in D}\left((\varphi^{-\frac{1}{k+p-1}})_{ii}+\varphi^{-\frac{1}{k+p-1}}\right)\leq 0$$
by assuming $\left(\varphi^{-\frac{1}{k+p-1}}\right)_{ii}+\varphi^{-\frac{1}{k+p-1}}\leq0$. But at the minimum point of $\varphi^{-\frac{1}{k+p-1}}$, the condition $\left(\varphi^{-\frac{1}{k+p-1}}\right)_{ii}+\varphi^{-\frac{1}{k+p-1}}\leq0$ is contradict with $\varphi>0$. This case does not occur.

\emph{Case 4:} $t>1, s> 0$, i.e., $p<1-k, q> k+1$.\\
Be analogue to Case 2, we have
\begin{eqnarray*}
 I&\leq& -\frac{p+k-1}{k}u^{-\frac{p-1}{k}}(u^2+|\nabla u|^2)^{-\frac{k+1-q}{2k}}\varphi^{-\frac{p-1}{k(k+p-1)}}\\
  &&\cdot\sum_{i\in D}\left((\varphi^{-\frac{1}{k+p-1}})_{ii}+\frac{2k+p-q}{k+p-1}\varphi^{-\frac{1}{k+p-1}}\right)\\
  &\leq&0
\end{eqnarray*}
by assuming $\left(\varphi^{-\frac{1}{k+p-1}}\right)_{ii}+\frac{2k+p-q}{k+p-1}\varphi^{-\frac{1}{k+p-1}}\leq0$. At the minimum point of $\varphi^{-\frac{1}{k+p-1}}$, we also need the condition $k+1<q<2k+p$, which is contradict with $p<1-k$. This case does not occur. 
\end{proof}

\section{The case $p>q$}

\subsection{The a priori estimates}

We derive  $C^0$ estimates at first. According to \eqref{rho} and the fact that $\max_{\mathbb{S}^n}\rho=\max_{\mathbb{S}^n}u$, then $\max_{\mathbb{S}^n}|\nabla u|^2\leq \max_{\mathbb{S}^n}\rho^2=\max_{\mathbb{S}^n}u^2$. Hence $C^1$ estimates can be obtained from $C^0$ estimates.
\begin{theorem}\label{c0}
Suppose $\varphi$ is a positive smooth function and $u\in C^2(\mathbb{S}^n)$ is a positive admissible solution of equation \eqref{G-eq}. Then for $p>q$,
$$\frac{C_n^k}{\max_{\mathbb{S}^n} \varphi}\leq u(x)^{p-q}\leq\frac{C_n^k}{\min_{\mathbb{S}^n} \varphi}, \quad \forall~x\in \mathbb{S}^n.$$
\end{theorem}
\begin{proof}
Assume that $\min_{\mathbb{S}^n}u(x)$ is attained at $x_0$, then at $x_0$ we get
$$|\nabla u|=0, \quad \nabla^2 u\geq 0.$$
Hence $\nabla^2u+uI\geq uI$, i.e., $b_{ii}\geq u$ for $i=1,2,\cdots,n$. Then,
$$\varphi(u^2+|\nabla u|^2)^{\frac{k+1-q}{2}}u^{p-1}=\sigma_k(\nabla^2u+uI)\geq C_n^ku^k.$$
Therefore
$$u(x_0)^{p-q}\geq\frac{C_n^k}{\varphi(x_0)}\geq\frac{C_n^k}{\max_{\mathbb{S}^n}\varphi}.$$
Similarly, assume that $\max_{\mathbb{S}^n}u(x)$ is attained at $x_1$, then at $x_1$ we get
$$|\nabla u|=0, \quad \nabla^2 u\leq 0.$$
Hence $\nabla^2u+uI\leq uI$, i.e., $b_{ii}\leq u$ for $i=1,2,\cdots,n$. Then,
$$\varphi(u^2+|\nabla u|^2)^{\frac{k+1-q}{2}}u^{p-1}=\sigma_k(\nabla^2u+uI)\leq C_n^ku^k.$$
Therefore
$$u(x_1)^{p-q}\leq\frac{C_n^k}{\varphi(x_1)}\leq\frac{C_n^k}{\min_{\mathbb{S}^n}\varphi}.$$
\end{proof}

Inspired by the work of Chu \cite{Chu21},
 we consider the following general Christoffel-Minkowski type equations with general right hand functions:
\begin{equation}\label{K-eq}
  \sigma_k(u_{ij}+u\delta_{ij})=f(x, u, \nabla u),\quad on~ \mathbb{S}^{n},
\end{equation}
and we derive $C^2$ estimates for spherical convex solutions
 of equation \eqref{K-eq}.
\begin{theorem}\label{c2-main}
Let $f: \mathbb{S}^n \times  \mathbb{R}^{1} \times \mathbb{R}^{n} \rightarrow \mathbb{R}$ be a positive smooth function, and $u:\mathbb{S}^n\rightarrow \mathbb{R}$ be a positive spherical convex solution of equation \eqref{K-eq}. Then there exists a constant $C$ depending only on $n, k, \inf u, \inf f, \|u\|_{C^1}$ and $\|f\|_{C^2}$ such that
$$\max_{\mathbb{S}^n}|\nabla^2u|\leq C.$$
\end{theorem}

\begin{proof}
We consider the auxiliary function
$$Q=\log \lambda_1+\varphi(|\nabla u|^2)+\psi(u),$$
where $\lambda_1=\lambda_{max}(\nabla^2u+uI)$ is the largest eigenvalue of $\nabla^2u+uI$.

Define
$$\varphi(s)=-A\log(1-\frac{s}{2K}), \quad 0\leq s \leq K-1,$$
and
$$\psi(t)=-B\log(1+\frac{t}{2L}), \quad 0< t \leq L-1.$$
Here we set
$$K=\sup_{\mathbb{S}^n}|\nabla u|^2+1, \quad L=\sup_{\mathbb{S}^n}|u|+1,\quad B:=3L\Lambda$$
and $A, B, \Lambda>1$ are large constants to be determined later. Clearly, $\varphi$ satisfies
$$\frac{A}{2K}\leq\varphi'\leq \frac{A}{K}, \quad \varphi''=\frac{1}{A}(\varphi')^2,$$
and $\psi$ satisfies
$$\Lambda\leq-\psi'< \frac{3\Lambda}{2}, \quad \psi''= \frac{1}{B}(\psi')^2.$$
Assume that $Q$ attains its maximum at $x_0\in\mathbb{S}^n$. Denote $\lambda(\{b_{ij}\})=(\lambda_1,\lambda_2,\cdots,\lambda_n)$ are eigenvalues of $\nabla^2u+uI$, we can choose a local orthonormal frame $\{e_1,e_2,\cdots,e_n\}$ near $x_0$ such that
$$\lambda_i=\delta_{ij}b_{ij}, \quad \lambda_1\geq \lambda_2\geq \cdots \geq\lambda_n, \quad \mbox{at}~ x_0.$$
Since $Q$ may be not smooth at $x_0$ when the eigenspace of $\lambda_1$ has dimension strictly larger than 1, we need to perturb $b_{ij}$ by a diagonal matrix $T$ with $T_{11}=0$, $T_{22}=T_{33}=\cdots=T_{nn}=\frac{1}{3}$ at $x_0$.
Define the matrix by $\widetilde{b}_{ij}=b_{ij}-T_{ij}$ and denote its eigenvalues by $\widetilde{\lambda}_1\geq\widetilde{\lambda}_2\geq \cdots \geq\widetilde{\lambda}_n$. Then it follows that $\lambda_1\geq\widetilde{\lambda}_1$ near $x_0$ and
\begin{eqnarray*}
  \widetilde{\lambda}_i=
  \begin{cases}
  \lambda_1, ~&\mbox{if}~i=1,\\
  \lambda_i-T_{ii}, ~ &\mbox{if}~i>1,
  \end{cases}
  \quad \mbox{at}~x_0.
\end{eqnarray*}
Thus $\widetilde{\lambda}_1>\widetilde{\lambda}_2$ at $x_0$, then $\widetilde{\lambda}_1$ is smooth at $x_0$. We consider the new test function
$$\widetilde{Q}=\log \widetilde{\lambda}_1+\varphi(|\nabla u|^2)+\psi(u).$$
It still achieves a local maximum at $x_0$. Hence at $x_0$, we have
\begin{equation}\label{Qi}
  0=\widetilde{Q}_i=\frac{\widetilde{\lambda}_{1,i}}{\lambda_1}+\varphi'\nabla_i(|\nabla u|^2)+\psi'u_i,
\end{equation}
and
\begin{eqnarray}\label{Qii}
  \nonumber0\geq\sigma_k^{ii}\widetilde{Q}_{ii}&=&\sigma_k^{ii}(\log \widetilde{\lambda}_1)_{ii}+\varphi''\sigma_k^{ii}\left(\nabla_i(|\nabla u|^2)\right)^2\\
  &&+\varphi'\sigma_k^{ii}\nabla_{ii}(|\nabla u|^2)+\psi''\sigma_k^{ii}u_i^2+\psi'\sigma_k^{ii}u_{ii}.
\end{eqnarray}
We divide our proof into four steps. For convenience, we will use a unified notation $C$ to denote a constant depending on $n, k, \inf u, \inf f, \|u\|_{C^1}, \|f\|_{C^2}$ and the perturbation $T$.

$\mathbf{Step ~1:}$ We show that at $x_0$,
\begin{eqnarray}\label{step1}
 \nonumber 0&\geq& -\frac{\sigma_k^{pp,qq}b_{pp1}b_{qq1}}{\lambda_1}+2\sum_{p>1}\frac{\sigma_k^{11,pp}b_{11p}^2}{\lambda_1}+2\sum_{p>1}\frac{\sigma_k^{11}\widetilde{b}_{1p1}^2}
  {\lambda_1(\lambda_1-\widetilde{\lambda}_p)}+2\sum_{p>1}\frac{\sigma_k^{pp}\widetilde{b}_{1pp}^2}{\lambda_1(\lambda_1-\widetilde{\lambda}_p)}\\
  \nonumber &&+2\varphi'\sigma_k^{ii}u_{ii}^2
  -\frac{\sigma_k^{ii}\widetilde{b}_{11i}^2}{\lambda_1^2}+\varphi''\sigma_k^{ii}(\nabla_i(|\nabla u|^2))^2+\psi''\sigma_k^{ii}u_i^2-\frac{CKL^2}{\lambda_1}\\
  &&-C\lambda_1+\left(\frac{\Lambda}{C}-CA-\frac{C}{\lambda_1}\right)\sum_i\sigma_k^{ii}-CK\Lambda-CA.
\end{eqnarray}

The following calculations are all at $x_0$. First, we deal with the term $\sigma_k^{ii}(\log\widetilde{\lambda}_1)_{ii}$ in \eqref{Qii}.
Thus
$$\widetilde{\lambda}_{1,i}=\frac{\partial\widetilde{\lambda}_1}{\partial \widetilde{b}_{pq}}\widetilde{b}_{pqi}=\delta_{1p}\delta_{1q}\widetilde{b}_{pqi}=\widetilde{b}_{11i}=b_{11i}-T_{11i},$$
\begin{eqnarray*}
\widetilde{\lambda}_{1,ii}&=&\frac{\partial\widetilde{\lambda}_1}{\partial\widetilde{b}_{pq}}\widetilde{b}_{pqii}
+\frac{\partial^2\widetilde{\lambda}_1}{\partial\widetilde{b}_{pq}\partial\widetilde{b}_{rs}}\widetilde{b}_{pqi}\widetilde{b}_{rsi}\\
&=&\delta_{1p}\delta_{1q}\widetilde{b}_{pqii}+\left[(1-\delta_{1p})\frac{\delta_{1q}\delta_{1r}\delta_{ps}}
{\widetilde{\lambda}_1-\widetilde{\lambda}_p}+(1-\delta_{1r})\frac{\delta_{1s}\delta_{1p}\delta_{qr}}
{\widetilde{\lambda}_1-\widetilde{\lambda}_r}\right]\widetilde{b}_{pqi}\widetilde{b}_{rsi}\\
&=&\widetilde{b}_{11ii}+2\sum_{p>1}\frac{\widetilde{b}_{1pi}^2}{\lambda_1-\widetilde{\lambda}_p}=b_{11ii}-T_{11ii}+2\sum_{p>1}\frac{\widetilde{b}_{1pi}^2}
{\lambda_1-\widetilde{\lambda}_p}.
\end{eqnarray*}
Then
\begin{eqnarray}\label{sig0}
 \nonumber \sigma_k^{ii}(\log\widetilde{\lambda}_1)_{ii}&=& \frac{\sigma_k^{ii}\widetilde{\lambda}_{1,ii}}{\widetilde{\lambda}_1}-\frac{\sigma_k^{ii}\widetilde{\lambda}_{1,i}^2}{\widetilde{\lambda}_1^2}\\
  &\geq& \frac{\sigma_k^{ii}b_{11ii}}{\lambda_1}-\frac{C\sum_i\sigma_k^{ii}}{\lambda_1}+2\sum_{p>1}\frac{\sigma_k^{ii}\widetilde{b}_{1pi}^2}{\lambda_1
  (\lambda_1-\widetilde{\lambda}_p)}
  -\frac{\sigma_k^{ii}\widetilde{b}_{11i}^2}{\lambda_1^2}.
\end{eqnarray}
By the Ricci identity $b_{11ii}=b_{ii11}+b_{11}-b_{ii}$, we have
\begin{equation}\label{sig1}
  \sigma_k^{ii}b_{11ii} = \sigma_k^{ii}b_{ii11}+b_{11}\sum_i\sigma_k^{ii}-\sigma_k^{ii}b_{ii}=\sigma_k^{ii}b_{ii11}+b_{11}\sum_i\sigma_k^{ii}-kf.
\end{equation}
Differentiating \eqref{K-eq} twice and recalling that $b_{ij}=u_{ij}+u\delta_{ij}$, we derive
\begin{eqnarray}\label{sig2}
  \nonumber\sigma_k^{ii}b_{ii11} &\geq&-\sigma_k^{ij,pq}b_{ij1}b_{pq1}-CK-CKu_{11}-Cu_{11}^2+f_{p_l}u_{l11} \\
  &\geq& -\sigma_k^{ij,pq}b_{ij1}b_{pq1}-CK-CKb_{11}-Cb_{11}^2+f_{p_l}b_{l11}.
\end{eqnarray}
Combining with \eqref{Qi} and the fact that $b_{ijk}=b_{ikj}$, then
\begin{eqnarray}\label{sig3}
 \nonumber \sum_lf_{p_l}\frac{b_{l11}}{b_{11}}&=&\sum_lf_{p_l}\frac{b_{11l}}{b_{11}}=\sum_lf_{p_l}\frac{\widetilde{b}_{11l}+T_{11l}}{b_{11}}  \\
 \nonumber &\geq&-\frac{C}{b_{11}}+\sum_lf_{p_l}(-2\varphi'u_lu_{ll}-\psi'u_l)\\
  &\geq&-2\varphi'\sum_lf_{p_l}u_lu_{ll}+CK\psi'-\frac{C}{\lambda_1}.
\end{eqnarray}

By \eqref{sig0}-\eqref{sig3}, we get
\begin{eqnarray}\label{log}
 \nonumber \sigma_k^{ii}(\log\widetilde{\lambda}_1)_{ii} &\geq&-\frac{\sigma_k^{ij,pq}b_{ij1}b_{pq1}}{\lambda_1}+2\sum_{p>1}\frac{\sigma_k^{ii}\widetilde{b}_{1pi}^2}{\lambda_1(\lambda_1-\widetilde{\lambda}_p)}
-2\varphi'\sum_lf_{p_l}u_lu_{ll}\\
&&+\sum_i\sigma_k^{ii}-\frac{\sigma_k^{ii}\widetilde{b}_{11i}^2}{\lambda_1^2}-CK\Lambda-\frac{C\sum_i\sigma_k^{ii}}{\lambda_1}-\frac{CKL^2}{\lambda_1}-C\lambda_1.
\end{eqnarray}
Then for the term $\varphi'\sigma_k^{ii}\nabla_{ii}(|\nabla u|^2)$ in \eqref{Qii}, we have
\begin{eqnarray}\label{nu1}
 \nonumber \varphi'\sigma_k^{ii}\nabla_{ii}(|\nabla u|^2)&=& 2\varphi'\sigma_k^{ii}u_{ii}^2+2\varphi'\sigma_k^{ii}\sum_lu_lu_{lii} \\
 \nonumber &=& 2\varphi'\sigma_k^{ii}u_{ii}^2+2\varphi'\sum_lu_l(f_l+f_zu_l+f_{p_l}u_{ll})-2\varphi'\sum_i\sigma_k^{ii}\sum_lu_l^2\\
  &\geq&2\varphi'\sigma_k^{ii}u_{ii}^2-CA\sum_i\sigma_k^{ii}-CA+2\varphi'\sum_lf_{p_l}u_lu_{ll}.
\end{eqnarray}
For the term $\psi'\sigma_k^{ii}u_{ii}$ in \eqref{Qii}, we derive
\begin{equation}\label{au}
\psi'\sigma_k^{ii}u_{ii}=\psi'\sigma_k^{ii}b_{ii}-\psi'u\sum_i\sigma_k^{ii}\geq-C\Lambda+\frac{\Lambda}{C}\sum_i\sigma_k^{ii},
\end{equation}
here we use that $u$ has a positive lower bound.

Substituting \eqref{log}-\eqref{au} into \eqref{Qii}, then
\begin{eqnarray}\label{ste}
 \nonumber0 &\geq& -\frac{\sigma_k^{ij,pq}b_{ij1}b_{pq1}}{\lambda_1}+2\sum_{p>1}\frac{\sigma_k^{ii}\widetilde{b}_{1pi}^2}{\lambda_1(\lambda_1-\widetilde{\lambda}_p)}
  +2\varphi' \sigma_k^{ii}u_{ii}^2-\frac{\sigma_k^{ii}\widetilde{b}_{11i}^2}{\lambda_1^2}\\
  \nonumber&&+\varphi''\sigma_k^{ii}(\nabla_i(|\nabla u|^2))^2+\psi''\sigma_k^{ii}u_i^2+\left(\frac{\Lambda}{C}-CA- \frac{C}{\lambda_1}\right)\sum_i\sigma_k^{ii}\\
  &&-\frac{CKL^2}{\lambda_1}-C\lambda_1-CK\Lambda -CA.
\end{eqnarray}

Since
\begin{eqnarray}\label{fra}
 && -\frac{\sigma_k^{ij,pq}b_{ij1}b_{pq1}}{\lambda_1}+2\sum_{p>1}\frac{\sigma_k^{ii}\widetilde{b}_{1pi}^2}{\lambda_1(\lambda_1-\widetilde{\lambda}_p)}\\
 \nonumber &\geq&-\frac{\sigma_k^{pp,qq}b_{pp1}b_{qq1}}{\lambda_1}-2\sum_{p>1}\frac{\sigma_k^{1p,p1}b_{11p}^2}{\lambda_1}
  +2\sum_{p>1}\frac{\sigma_k^{11}\widetilde{b}_{1p1}^2}{\lambda_1(\lambda_1-\widetilde{\lambda}_p)}
  +2\sum_{p>1}\frac{\sigma_k^{pp}\widetilde{b}_{1pp}^2}{\lambda_1(\lambda_1-\widetilde{\lambda}_p)} \\
 \nonumber &=& -\frac{\sigma_k^{pp,qq}b_{pp1}b_{qq1}}{\lambda_1}+2\sum_{p>1}\frac{\sigma_k^{11,pp}b_{11p}^2}{\lambda_1}
  +2\sum_{p>1}\frac{\sigma_k^{11}\widetilde{b}_{1p1}^2}{\lambda_1(\lambda_1-\widetilde{\lambda}_p)}
  +2\sum_{p>1}\frac{\sigma_k^{pp}\widetilde{b}_{1pp}^2}{\lambda_1(\lambda_1-\widetilde{\lambda}_p)},
\end{eqnarray}
where we use $-\sigma_k^{1p,p1}=\sigma_k^{11,pp}$.

Then inserting \eqref{fra} into \eqref{ste}, we obtain \eqref{step1}.

$\mathbf{Step ~2:}$ Without loss of generality we assume that $\lambda_1\geq1$. Then we claim
\begin{equation*}
  2\sum_{p>1}\frac{\sigma_k^{11,pp}b_{11p}^2}{\lambda_1}+2\sum_{p>1}\frac{\sigma_k^{11}\widetilde{b}_{1p1}^2}{\lambda_1(\lambda_1-\widetilde{\lambda}_p)}
  \geq \sum_{p>1}\frac{\sigma_k^{pp}\widetilde{b}_{11p}^2}{\lambda_1^2}-C\sum_i\sigma_k^{ii}.
\end{equation*}
Define
$$I=\{i\in\{2,3,\cdots,n\}| \lambda_i=\lambda_1\}.$$
Note that $\widetilde{\lambda}_p=\lambda_p-\frac{1}{3}$ and $\lambda_p\geq0$ with $p>1$, then
\begin{equation}\label{wid}
  \lambda_1-\widetilde{\lambda}_p=\lambda_1-\lambda_p+\frac{1}{3}\leq\lambda_1+\frac{1}{3}\leq\frac{4}{3}\lambda_1.
\end{equation}
Combining with \eqref{wid} and the fact  $\frac{\sigma_k^{pp}-\sigma_k^{11}}{\lambda_1-\lambda_p}=\sigma_k^{11,pp}$, we have
\begin{eqnarray}\label{su}
  \nonumber\sum_{p>1}\frac{\sigma_k^{pp}b_{11p}^2}{\lambda_1^2}&=&\sum_{p\in I}\frac{\sigma_k^{pp}b_{11p}^2}{\lambda_1^2}+\sum_{p\notin I}\frac{\sigma_k^{11}b_{11p}^2}{\lambda_1^2}+\sum_{p\notin I}\frac{(\sigma_k^{pp}-\sigma_k^{11})b_{11p}^2}{\lambda_1^2} \\
  \nonumber&\leq&\sum_{p\in I}\frac{4\sigma_k^{11}b_{11p}^2}{3\lambda_1(\lambda_1-\widetilde{\lambda}_p)}+\sum_{p\notin I}\frac{4\sigma_k^{11}b_{11p}^2}{3\lambda_1(\lambda_1-\widetilde{\lambda}_p)}+\sum_{p\notin I}\frac{(\sigma_k^{pp}-\sigma_k^{11})b_{11p}^2}{\lambda_1(\lambda_1-\lambda_p)}\\
  &\leq&\frac{4}{3}\sum_{p>1}\frac{\sigma_k^{11}b_{11p}^2}{\lambda_1(\lambda_1-\widetilde{\lambda}_p)}+\sum_{p>1}\frac{\sigma_k^{11,pp}b_{11p}^2}
  {\lambda_1}.
\end{eqnarray}
Due to \eqref{su}, we get
\begin{eqnarray*}
  \nonumber\sum_{p>1}\frac{\sigma_k^{pp}\widetilde{b}_{11p}^2}{\lambda_1^2}&=&\sum_{p>1}\frac{\sigma_k^{pp}(b_{11p}-T_{11p})^2}{\lambda_1^2}\\
  \nonumber&\leq&\frac{5}{4}\sum_{p>1}\frac{\sigma_k^{pp}b_{1p1}^2}{\lambda_1^2}+C\sum_i\sigma_k^{ii}\\
  \nonumber&\leq&\frac{5}{3}\sum_{p>1}\frac{\sigma_k^{11}b_{1p1}^2}{\lambda_1(\lambda_1-\widetilde{\lambda}_p)}+\frac{5}{4}\sum_{p>1}
  \frac{\sigma_k^{11,pp}b_{11p}^2}{\lambda_1}+C\sum_i\sigma_k^{ii}\\
  \nonumber&=&\frac{5}{3}\sum_{p>1}\frac{\sigma_k^{11}(\widetilde{b}_{1p1}+T_{1p1})^2}{\lambda_1(\lambda_1-\widetilde{\lambda}_p)}+\frac{5}{4}\sum_{p>1}
  \frac{\sigma_k^{11,pp}b_{11p}^2}{\lambda_1}+C\sum_i\sigma_k^{ii}\\
  &\leq&2\sum_{p>1}\frac{\sigma_k^{11}\widetilde{b}_{1p1}^2}{\lambda_1(\lambda_1-\widetilde{\lambda}_p)}+2\sum_{p>1}\frac{\sigma_k^{11,pp}b_{11p}^2}
  {\lambda_1}+C\sum_i\sigma_k^{ii}.
\end{eqnarray*}

$\mathbf{Step ~3:}$ We show that for $\varepsilon, \delta\in(0,\frac{1}{4})$ and $1\leq l\leq k-1$, there exists a constant $\delta'$ depending on $\varepsilon,\delta,n,k,\|f\|_{C^1}$ and $\inf f$ such that if $\lambda_l\geq \delta \lambda_1$, $\lambda_{l+1}\leq \delta'\lambda_1$, then
$$(1-4\varepsilon)\frac{\sigma_k^{11}\widetilde{b}_{111}^2}{\lambda_1^2}\leq-\frac{\sigma_k^{pp,qq}b_{pp1}b_{qq1}}{\lambda_1}
+2\sum_{p>1}\frac{\sigma_k^{pp}\widetilde{b}_{1pp}^2}{\lambda_1(\lambda_1-\widetilde{\lambda}_p)}+CK\lambda_1+C\left(\frac{(1-2\varepsilon)^2}{2\varepsilon \lambda_1^2} +1\right)\sum_i\sigma_k^{ii}.$$

By Lemma 7 in \cite{Guan15}, we derive
\begin{eqnarray*}
 \nonumber &&-\frac{\sigma_k^{pp,qq}b_{pp1}b_{qq1}}{\lambda_1}+\frac{(\sum_p\sigma_k^{pp}b_{pp1})^2}{\lambda_1\sigma_k} \\
 \nonumber &\geq& \frac{\sigma_k}{\lambda_1\sigma_l^2}[(\sum_p\sigma_l^{pp}b_{pp1})^2-\sigma_l\sigma_l^{pp,qq}b_{pp1}b_{qq1}] \\
  &=& \frac{\sigma_k}{\lambda_1\sigma_l^2}[\sum_p(\sigma_l^{pp}b_{pp1})^2+\sum_{p\neq q}(\sigma_l^{pp}\sigma_l^{qq}-\sigma_l\sigma_l^{pp,qq})
  b_{pp1}b_{qq1}].
\end{eqnarray*}
Differentiating \eqref{K-eq} once, we have
$$\sum_p\sigma_k^{pp}b_{pp1}=f_1+f_zu_1+f_{p_1}u_{11}.$$
Then
$$\frac{(\sum_p\sigma_k^{pp}b_{pp1})^2}{\sigma_k\lambda_1}=\frac{(f_1+f_zu_1+f_{p_1}u_{11})^2}{\lambda_1f}\leq CK\lambda_1.$$
Hence
\begin{eqnarray}\label{pp}
 \nonumber &&-\frac{\sigma_k^{pp,qq}b_{pp1}b_{qq1}}{\lambda_1}+CK\lambda_1 \\
  &\geq&\frac{\sigma_k}{\lambda_1\sigma_l^2}[\sum_p(\sigma_l^{pp}b_{pp1})^2+\sum_{p\neq q}(\sigma_l^{pp}\sigma_l^{qq}-\sigma_l\sigma_l^{pp,qq})b_{pp1}b_{qq1}].
\end{eqnarray}
In order to deal with $\sum_{p\neq q}(\sigma_l^{pp}\sigma_l^{qq}-\sigma_l\sigma_l^{pp,qq})b_{pp1}b_{qq1}$, we claim that
\begin{equation}\label{neq}
\sum_{p\neq q}(\sigma_l^{pp}\sigma_l^{qq}-\sigma_l\sigma_l^{pp,qq})b_{pp1}b_{qq1}
   \geq-\varepsilon\sum_{p\leq l}(\sigma_l^{pp}b_{pp1})^2-\frac{C}{\varepsilon}\sum_{p>l}(\sigma_l^{pp}b_{pp1})^2.
\end{equation}
The proof of \eqref{neq} is similar to \cite{Chu21} and we omit here. Then by \eqref{pp}-\eqref{neq}, we get
\begin{eqnarray}\label{rem}
 \nonumber (1-\varepsilon)\frac{\sigma_k(\sigma_l^{11})^2b_{111}^2}{\lambda_1\sigma_l^2} &\leq& (1-\varepsilon)\frac{\sigma_k}{\lambda_1\sigma_l^2}\sum_
  {p\leq l}(\sigma_l^{pp}b_{pp1})^2 \\
  &\leq& \frac{C\sigma_k}{\varepsilon\lambda_1\sigma_l^2}\sum_{p>l}(\sigma_l^{pp}b_{pp1})^2-\frac{\sigma_k^{pp,qq}}{\lambda_1}b_{pp1}b_{qq1}+CK\lambda_1.
\end{eqnarray}
For the term $(1-\varepsilon)\frac{\sigma_k(\sigma_l^{11})^2b_{111}^2}{\lambda_1\sigma_l^2}$ in \eqref{rem}, using $\lambda_{l+1}\leq\delta'\lambda_1$ and $\lambda_i\geq0$ for all $i$, we have
$$\frac{\sigma_k}{\lambda_1\sigma_k^{11}}=\frac{\lambda_1\sigma_k^{11}+\sigma_k(\lambda|1)}{\lambda_1\sigma_k^{11}}\geq1,$$
$$\frac{\lambda_1\sigma_l^{11}}{\sigma_l}=1-\frac{\sigma_l(\lambda|1)}{\sigma_l}\geq1-\frac{C\lambda_2\cdots\lambda_{l+1}}{\lambda_1
\cdots\lambda_l}=1-\frac{C\lambda_{l+1}}{\lambda_1}\geq1-C\delta'.$$
Thus  choosing $\delta'\leq\frac{1}{C}\left(1-\sqrt{\frac{1-2\varepsilon}{1-\varepsilon}}\right)$,
\begin{eqnarray}\label{var}
  \nonumber(1-\varepsilon)\frac{\sigma_k(\sigma_l^{11})^2b_{111}^2}{\lambda_1\sigma_l^2}
  &=&(1-\varepsilon)\frac{\sigma_k^{11}b_{111}^2}{\lambda_1^2}\frac{\sigma_k}{\lambda_1\sigma_k^{11}}\left(\frac{\lambda_1\sigma_l^{11}}{\sigma_l}\right)^2\\
 \nonumber &\geq&(1-\varepsilon)(1-C\delta')^2\frac{\sigma_k^{11}b_{111}^2}{\lambda_1^2}\\
 \nonumber &\geq&(1-2\varepsilon)\frac{\sigma_k^{11}(\widetilde{b}_{111}+T_{111})^2}{\lambda_1^2}\\
  &\geq& (1-4\varepsilon)\frac{\sigma_k^{11}\widetilde{b}_{111}^2}{\lambda_1^2}-C\frac{(1-2\varepsilon)^2}{2\varepsilon\lambda_1^2}\sum_i\sigma_k^{ii}.
\end{eqnarray}
For the term $\frac{C\sigma_k}{\varepsilon\lambda_1\sigma_l^2}\sum_{p>l}(\sigma_l^{pp}b_{pp1})^2$ in \eqref{rem}, by $\lambda_l\geq\delta\lambda_1$ and $\lambda_i\geq0$ for all $i$, we get
$$\frac{\sigma_l^{pp}}{\sigma_l}\leq\frac{C\lambda_1\cdots\lambda_{l-1}}{\lambda_1\cdots\lambda_l}\leq\frac{C}{\lambda_l}
\leq\frac{C}{\delta\lambda_1},$$
which implies
\begin{eqnarray*}
  \nonumber\frac{C\sigma_k}{\varepsilon\lambda_1\sigma_l^2}\sum_{p>l}(\sigma_l^{pp}b_{pp1})^2 &=& \frac{C}{\varepsilon}\sum_{p>l}\left(\frac{\sigma_l^{pp}}{\sigma_l}\right)^2\frac{\sigma_kb_{pp1}^2}{\lambda_1} \\
  &\leq&\frac{C}{\varepsilon\delta^2}\sum_{p>l}\left(\frac{\sigma_k}{\lambda_1}\right)\left(\frac{b_{pp1}}{\lambda_1}\right)^2.
\end{eqnarray*}
For the term $\frac{\sigma_k}{\lambda_1}$, we know that
$$\frac{\sigma_k}{\lambda_1}\leq\frac{\delta'\sigma_k}{\lambda_p}\leq \frac{C\delta'\lambda_1\cdots\lambda_k}{\lambda_p}\leq C\delta'\sigma_k^{pp},$$
for $l<p\leq k$ and
$$\frac{\sigma_k}{\lambda_1}\leq\frac{\delta'\sigma_k}{\lambda_k}\leq C\delta'\lambda_1\cdots\lambda_{k-1}\leq C\delta'\sigma_k^{pp},$$
for $p>k\geq l+1$. Choosing $\delta'\leq\frac{\varepsilon\delta^2}{C}$, by \eqref{wid} we derive
\begin{eqnarray}\label{kpp}
 \nonumber \frac{C\sigma_k}{\varepsilon\lambda_1\sigma_l^2}\sum_{p>l}(\sigma_l^{pp}b_{pp1})^2&\leq&\frac{C\delta'}{\varepsilon\delta^2}
 \sum_{p>l}\frac{\sigma_k^{pp}b_{1pp}^2}{\lambda_1^2}\\
  \nonumber&\leq&\sum_{p>l}\frac{\sigma_k^{pp}(\widetilde{b}_{1pp}+T_{1pp})^2}{\lambda_1^2}\\
  \nonumber&\leq&\frac{3}{2}\sum_{p>l}\frac{\sigma_k^{pp}\widetilde{b}_{1pp}^2}{\lambda_1^2}+\frac{C}{\lambda_1^2}\sum_i\sigma_k^{ii}\\
  &\leq&2\sum_{p>1}\frac{\sigma_k^{pp}\widetilde{b}_{1pp}^2}{\lambda_1(\lambda_1-\widetilde{\lambda}_p)}+\frac{C}{\lambda_1^2}\sum_i\sigma_k^{ii}.
\end{eqnarray}
Combining with \eqref{rem}-\eqref{kpp}, we obtain
\begin{equation*}
  (1-4\varepsilon)\frac{\sigma_k^{11}\widetilde{b}_{111}^2}{\lambda_1^2}
  \leq-\frac{\sigma_k^{pp,qq}}{\lambda_1}b_{pp1}b_{qq1}+2\sum_{p>1}\frac{\sigma_k^{pp}\widetilde{b}_{1pp}^2}{\lambda_1(\lambda_1-\widetilde{\lambda}_p)} +CK\lambda_1+C \left(\frac{(1-2\varepsilon)^2}{2\varepsilon \lambda_1^2} +1\right)\sum_i\sigma_k^{ii}.
\end{equation*}

$\mathbf{Step ~4:}$
We need to prove the following lemma.

\begin{lemma}\label{lem}
For any $\delta\in(0, \frac{1}{4})$ and $1\leq l\leq k-1$, there exist constants $\delta'$ and C depending on $\delta, n, k, \inf u, \inf f, \|u\|_{C^1}$ and $\|f\|_{C^2}$ such that if $\lambda_l\geq \delta\lambda_1$ and $\lambda_{l+1} \leq \delta' \lambda_1$, then
$$\lambda_1 \leq C.$$
\end{lemma}

\emph{Proof of Lemma \ref{lem}.}
Assume that $A\leq \Lambda<B$. According to \eqref{Qi} we get
\begin{eqnarray}\label{cont}
 \nonumber \varphi''\sigma_k^{ii}(\nabla_i|\nabla u|^2)^2+\psi''\sigma_k^{ii}u_i^2&=& \frac{1}{A}\sigma_k^{ii}(\varphi'\nabla_i|\nabla u|^2)^2+ \frac{1}{B}\sigma_k^{ii}(\psi'u_i)^2 \\
 \nonumber &\geq& \frac{1}{B} \left( \sigma_k^{ii}(\frac{\widetilde{b}_{11i}}{\lambda_1}+\psi'u_i)^2+\sigma_k^{ii}(\psi'u_i)^2 \right)\\
  &\geq&\frac{1}{2B}\frac{\sigma_k^{ii}\widetilde{b}_{11i}^2}{\lambda_1^2}.
\end{eqnarray}
By Step 1- Step 3, \eqref{cont} and the fact that $\lambda_1=b_{11}\leq C\sigma_k^{11}b_{11}^2\leq C\sigma_k^{ii}b_{ii}^2$, we derive
\begin{eqnarray*}
   0&\geq& \left(\frac{1}{2B}-4\varepsilon\right)\frac{\sigma_k^{11}\widetilde{b}_{111}^2}{\lambda_1^2}+\left(\frac{\Lambda}{C}-CAL^2-\frac{C}{\lambda_1}
   -\frac{C(1-2\varepsilon)^2}{2\varepsilon\lambda_1^2}\right)\sum_i\sigma_k^{ii}\\
   &&+\left(\frac{A}{CK}-CK\right)\lambda_1-CKL^2-CK\Lambda-CA.
\end{eqnarray*}
Hence choosing $B>\Lambda\gg A\gg1, \varepsilon\ll 1$ and $\lambda_1$ large enough,
we have
$$\lambda_1\leq CKL^2+CK\Lambda+CA,$$
the Lemma \ref{lem}
is finished.

Then we continue to complete the proof of Theorem \ref{c2-main}.
Set $\delta_1=\frac{1}{5}$, by Lemma \ref{lem} there exists $\delta_2$ such that if $\lambda_2\leq\delta_2\lambda_1$, then $\lambda_1\leq C$. If $\lambda_2>\delta_2\lambda_1$, using Lemma \ref{lem} again, there exists $\delta_3$ such that if $\lambda_3\leq\delta_3\lambda_1$, then $\lambda_1\leq C$. Repeating the above argument, we get $\lambda_1\leq C$ or $\lambda_k>\delta_k\lambda_1$. In the latter case, since $\lambda_1\geq\lambda_2\geq\cdots\geq\lambda_k>\delta_k\lambda_1$ and $\lambda_i\geq0$ for all $i$, then
$$\delta_k^k\lambda_1^k<\lambda_1\cdots\lambda_k\leq\sigma_k=f\leq C,$$
which also implies $\lambda_1\leq C$. Therefore we derive $\max_{\mathbb{S}^n}|\nabla^2u|\leq C$.
\end{proof}

\subsection{Existence and uniqueness}
 Based on the a priori estimates and constant rank theorem, we can establish the existence and uniqueness for strictly spherical convex solutions of equation \eqref{G-eq} by the continuity method. We will divide into two steps to complete the proof of Theorem \ref{exist-main} for the case $p>q$.

\textbf{Step 1:} Existence:
 As in \cite{GX18}, we consider the following equation
 \begin{equation}\label{kij}
   \sigma_k(u_{ij}+u\delta_{ij})=u^{p-1}(u^2+|\nabla u|^2)^{\frac{k+1-q}{2}}\varphi_t,\quad \forall~ 0 \leq t\leq 1,
 \end{equation}
where $\varphi_t=\left((1-t)(C_n^k)^{-\frac{1}{p-1+k}}+t\varphi^{-\frac{1}{p-1+k}}\right)^{-(p-1+k)}$.

Denote
$$S=\{t\in [0,1]|\mbox{equation}~\eqref{kij} ~\mbox{has a positive strictly spherical convex solution} ~u_t\}.$$
When $t=0$, we have $\varphi_0=C_n^k$ and it is clear that $u_0\equiv1$ is a positive strictly spherical convex solution of equation \eqref{kij}. Thus $S$ is non-empty.

Next we prove $S$ is open.
Since equation \eqref{kij} can be expressed as
$$\sigma_k^{\frac{1}{k}}=u^{\frac{p-1}{k}}(u^2+|\nabla u|^2)^{\frac{k+1-q}{2k}}\varphi_t^{\frac{1}{k}}.$$
Denote $\widetilde{F}=\sigma_k^{\frac{1}{k}}$, $\widetilde{F}^{ij}=\frac{\partial \widetilde{F}}{\partial b_{ij}}$ and $b_{ij}=u_{ij}+u\delta_{ij}$,
then the linearized operator is given by
\begin{eqnarray*}
  L_u(v)&=&\widetilde{F}^{ij}(v_{ij}+v\delta_{ij})-\frac{p-1}{k}u^{\frac{p-1}{k}-1}(u^2+|\nabla u|^2)^{\frac{k+1-q}{2k}}\varphi_t^{\frac{1}{k}}v\\
  &&-\frac{k+1-q}{k}u^{\frac{p-1}{k}+1}(u^2+|\nabla u|^2)^{\frac{k+1-q}{2k}-1}\varphi_t^{\frac{1}{k}}v\\
  &&-\frac{k+1-q}{k}u^{\frac{p-1}{k}}(u^2+|\nabla u|^2)^{\frac{k+1-q}{2k}-1}\varphi_t^{\frac{1}{k}}\sum_ku_kv_k.
\end{eqnarray*}
Let $v=uw$, we get
\begin{eqnarray*}
  L_u(v)&=&\widetilde{F}^{ij}(u_{ij}+u\delta_{ij})w+2\widetilde{F}^{ij}u_iw_j+u\widetilde{F}^{ij}w_{ij}\\
&&-wu^{\frac{p-1}{k}}(u^2+|\nabla u|^2)^{\frac{k+1-q}{2k}}\varphi_t^{\frac{1}{k}}\bigg(\frac{p-1}{k}+\frac{k+1-q}{k}\frac{u^2}{u^2+|\nabla u|^2}\\
&&+\frac{k+1-q}{k}\frac{|\nabla u|^2}{u^2+|\nabla u|^2}+\frac{k+1-q}{k}\frac{u\sum_ku_kw_k}{w\left(u^2+|\nabla u|^2\right)}\bigg)\\
&=&-wu^{\frac{p-1}{k}}(u^2+|\nabla u|^2)^{\frac{k+1-q}{2k}}\varphi_t^{\frac{1}{k}}\left(\frac{p-q}{k}+\frac{k+1-q}{k}\frac{u\sum_ku_kw_k}{w\left(u^2+|\nabla u|^2\right)}\right)\\
&&+2\widetilde{F}^{ij}u_iw_j+u\widetilde{F}^{ij}w_{ij}.
\end{eqnarray*}
Assume that $w$ attains its maximum at $x_0$, thus $w_i(x_0)=0$ and $w_{ij}(x_0)\leq0$. If $L_u(v)=0$, then at $x_0$,
$$0\leq -\frac{p-q}{k}wu^{\frac{p-1}{k}}(u^2+|\nabla u|^2)^{\frac{k+1-q}{2k}}\varphi_t^{\frac{1}{k}}.$$
 By the condition $p>q$, we have $\max_{\mathbb{S}^n}w\leq0$. Similarly, we can also derive $\min_{\mathbb{S}^n}w\geq0$.  Hence $w=0$, which implies $v=0$. Therefore $\emph{Ker}L_u=\{0\}$, i.e., the linearized operator $L_u$ is invertible. By the implicit function theorem, for each $t_0\in S$, there exists a neighborhood $\mathcal{N}$ of $t_0$ such that there exists a positive strictly spherical convex solution $u_t$ of equation \eqref{kij} for $t\in \mathcal{N}$. Hence $\mathcal{N}\subset S$ and $S$ is open.

We now prove $S$ is closed. Let $\{t_i\}_{i=1}^{\infty}\subset S$ be a sequence such that $t_i\rightarrow t_0$ and $u_{t_i}$ be a positive strictly spherical convex solution of equation \eqref{kij} for $t=t_i$. Based on the a priori estimates, Evans-Krylov and Schauder theory, we can get higher order estimates. Then there exists a subsequence still denote by $u_{t_i}$ converges to some function $u$, and $u$ is a positive solution of equation \eqref{kij} for $t=t_0$. Suppose $(u_{ij}+u\delta_{ij})$ is not positive definite, then $(u_{ij}+u\delta_{ij})$ is semi-positive definite. Since $\varphi$ satisfies Assumption \ref{cdt-091}, it is easy to verify that $\varphi_{t_0}$ also satisfies Assumption \ref{cdt-091}, then by constant rank theorem, $(u_{ij}+u\delta_{ij})$ must be positive definite, which implies a contradiction. Therefore $t_0\in S$ and $S$ is closed.

We conclude that $S=[0,1]$ and equation \eqref{kij} with $t=1$, which is equation \eqref{G-eq} has a positive strictly spherical convex solution.

\textbf{Step 2: } Uniqueness:
Let $u, \overline{u}$ be two admissible solutions of equation \eqref{G-eq}. Suppose $P=\frac{u}{\overline{u}}$ attains its maximum at $x_0\in\mathbb{S}^n$, then at $x_0$,
 $$0=\nabla\log P=\frac{\nabla u}{u}-\frac{\nabla\overline{u}}{\overline{u}},$$
 and
 \begin{eqnarray*}
   0 &\geq& \nabla^2\log P \\
   &=&\frac{\nabla^2u}{u}-\left(\frac{\nabla u}{u}\right)^2-\frac{\nabla^2\overline{u}}{\overline{u}}+\left(\frac{\nabla \overline{u}}{\overline{u}}\right)^2\\
   &=&\frac{\nabla^2u}{u}-\frac{\nabla^2\overline{u}}{\overline{u}}.
 \end{eqnarray*}
  So $\frac{\nabla^2u+uI}{u}\leq\frac{\nabla^2\overline{u}+\overline{u}I}{\overline{u}}$. Then $\sigma_k(\frac{\nabla^2u+uI}{u})\leq \sigma_k(\frac{\nabla^2\overline{u}+\overline{u}I}{\overline{u}})$. Therefore,
 \begin{eqnarray*}
   1 &=& \frac{\varphi(x_0)}{\varphi(x_0)}=\frac{(u^2+|\nabla u|^2)^{-\frac{k+1-q}{2}}u^{1-p}\sigma_k(\nabla^2u+uI)}{(\overline{u}^2+|\nabla \overline{u}|^2)^{-\frac{k+1-q}{2}}\overline{u}^{1-p}\sigma_k(\nabla^2\overline{u}+\overline{u}I)}(x_0) \\
   &\leq& \frac{(u^2+|\nabla u|^2)^{-\frac{k+1-q}{2}}u^{k-p+1}}{(\overline{u}^2+|\nabla \overline{u}|^2)^{-\frac{k+1-q}{2}}\overline{u}^{k-p+1}}(x_0)\leq P(x_0)^{q-p}.
 \end{eqnarray*}
 We have $\max_{\mathbb{S}^n}P=P(x_0)\leq 1$ with $p>q$. The treatment for $\min_{\mathbb{S}^n}P\geq 1$ is similar. To sum up, we get $u\equiv\overline{u}$.

\section{The case $p=q>1$}
In this section, we consider equation \eqref{G-eq} for the  case $p=q$. Inspired by  the technique in \cite{GL99, HMS04, CX22}, we  outline the arguments here with necessary modifications.
We  study the following equation
\begin{equation}\label{e1}
 \sigma_k(\nabla^2u+uI)=u^{p-1+\varepsilon}(u^2+|\nabla u|^2)^{\frac{k+1-p}{2}}\varphi(x), \quad on~ \mathbb{S}^{n},
\end{equation}
for any small $\varepsilon>0$.
\subsection{The a priori estimates}
Following the proof of Theorem \ref{c0}, we can easily derive $C^0$ estimates for equation \eqref{e1}.
\begin{theorem}\label{C0}
Suppose $\varphi$ is a positive smooth function and $u\in C^2(\mathbb{S}^n)$ is a positive admissible  solution of equation \eqref{e1}. Then for $p=q$,
$$\frac{C_n^k}{\max_{\mathbb{S}^n} \varphi}\leq u(x)^{\varepsilon}\leq\frac{C_n^k}{\min_{\mathbb{S}^n} \varphi}, \quad \forall~x\in \mathbb{S}^n.$$
\end{theorem}
\begin{proof}
Assume that $\min_{\mathbb{S}^n}u(x)$ is attained at $x_0$, then at $x_0$ we get
$$|\nabla u|=0, \quad \nabla^2 u\geq 0.$$
Hence $\nabla^2u+uI\geq uI$, i.e., $b_{ii}\geq u$ for $i=1,2,\cdots,n$. Thus
$$\varphi(u^2+|\nabla u|^2)^{\frac{k+1-p}{2}}u^{p-1+\varepsilon}=\sigma_k(\nabla^2u+uI)\geq C_n^ku^k.$$
Therefore
$$u(x_0)^{\varepsilon}\geq\frac{C_n^k}{\varphi(x_0)}\geq\frac{C_n^k}{\max_{\mathbb{S}^n}\varphi}.$$
Similarly, assume that $\max_{\mathbb{S}^n}u(x)$ is attained at $x_1$, we can also get $u(x_1)^{\varepsilon}\leq\frac{C_n^k}{\min_{\mathbb{S}^n}\varphi}$, the proof is completed.
\end{proof}

Based on $C^0$ estimates, we can obtain the following $C^1$ estimates.
\begin{theorem}\label{c1}
Suppose $\varphi$ is a positive smooth function and $u\in C^2(\mathbb{S}^n)$ is a positive admissible solution of equation \eqref{e1}. Then for $p=q>1$ there is a positive constant $C$ depending on $n, k, \min_{\mathbb{S}^n} \varphi$ and $\|\varphi\|_{C^1}$, but independent of $\varepsilon$ such that
\begin{equation}\label{e2}
  \frac{|\nabla u(x)|}{u(x)}\leq C, \quad \forall~x\in \mathbb{S}^n,
\end{equation}
and
\begin{equation}\label{e3}
  \frac{\max_{\mathbb{S}^n}u(x)}{\min_{\mathbb{S}^n}u(x)}\leq C.
\end{equation}
\end{theorem}
\begin{proof}
Let $v=\log u$, then equation \eqref{e1} becomes
\begin{equation}\label{e4}
  \sigma_k(v_iv_j+v_{ij}+\delta_{ij})=e^{\varepsilon v}(1+|\nabla v|^2)^{\frac{k+1-p}{2}}\varphi(x).
\end{equation}
Consider the test function $P=|\nabla v|^2$. Suppose $x_0$ is a maximum point of $P$ and we can choose a local orthonormal frame field such that at $x_0$
$$v_1=|\nabla v|>0, \quad \{v_{ij}\}_{2\leq i,j\leq n} \quad\mbox{is diagonal}.$$
Hence we have
\begin{equation}\label{e5}
  0=P_i=2v_lv_{li}=2v_1v_{1i},
\end{equation}
which means $v_{1i}=0$ for $i=1,2,\cdots,n$. So we get $\{v_{ij}\}_{1\leq i,j \leq n}$ is diagonal, therefore
$$\{a_{ij}\}:= v_{ij}+v_iv_j+\delta_{ij}=diag\{1+v_1^2, 1+v_{22}, \cdots, 1+v_{nn}\}:=diag\{\lambda_1,\cdots,\lambda_n\}.$$

In the following our calculations will be done at $x_0$. Denote $\sigma_k^{ij}=\frac{\partial\sigma_k(a_{ij})}{\partial a_{ij}}$, then
\begin{equation}\label{e6}
  0\geq\frac{1}{2}\sigma_k^{ii}P_{ii}=\sigma_k^{ii}v_{ii}^2+\sigma_k^{ii}v_lv_{lii}\geq v_1\sigma_k^{ii}v_{1ii}.
\end{equation}
By Ricci identity, we derive
\begin{equation}\label{e7}
  v_{jii}=v_{iji}=v_{iij}+v_sR_{siji}=v_{iij}+v_s(\delta_{sj}\delta_{ii}-\delta_{si}\delta_{ij})=v_{iij}+v_j-v_i\delta_{ij}.
\end{equation}
Putting \eqref{e7} into \eqref{e6} and differentiating equation \eqref{e1} once, we get
\begin{eqnarray}\label{e8}
  \nonumber0 &\geq&v_1\sigma_k^{ii}v_{ii1}+v_1^2\sum_{i=2}^n\sigma_k^{ii}\\
 \nonumber &=&e^{\varepsilon v}(1+v_1^2)^{\frac{k+1-p}{2}}(\varepsilon v_1^2\varphi+v_1\varphi_1)+v_1^2\sum_{i=2}^n\sigma_k^{ii}\\
  &\geq&e^{\varepsilon v}(1+v_1^2)^{\frac{k+1-p}{2}}v_1\varphi_1+v_1^2\sum_{i=2}^n\sigma_k^{ii}.
\end{eqnarray}
By Proposition \ref{sigma} and \ref{NM}, we have
$$\sum_{i=2}^n\sigma_k^{ii}=\sum_{i=2}^n\sigma_{k-1}(\lambda|i)=(n-k+1)\sigma_{k-1}-\sigma_{k-1}(\lambda|1)\geq(n-k)\sigma_{k-1}(\lambda),$$
and
$$\sigma_{k-1}(\lambda)\geq C_n^{k-1}\left(\frac{\sigma_k}{C_n^k}\right)^{\frac{k-1}{k}}=C_n^{k-1}(C_n^k)^{-\frac{k-1}{k}}[e^{\varepsilon v}(1+v_1^2)^{\frac{k+1-p}{2}}\varphi(x)]^{\frac{k-1}{k}}.$$
Thus \eqref{e8} implies
$$0\geq\frac{\varphi_1}{\varphi}+c(n,k)v_1[e^{\varepsilon v}(1+v_1^2)^{\frac{k+1-p}{2}}\varphi(x)]^{-\frac{1}{k}},$$
where $c(n,k)$ is a constant depending only on $n, k$. By Theorem \ref{C0}, we know that $|e^{\varepsilon v}|=|u^{\varepsilon}|$ has a uniform bound independent of $\varepsilon$. Hence there exists a positive constant $C$ depending on $n, k, \min_{\mathbb{S}^n} \varphi$ and $\|\varphi\|_{C^1}$, but independent of $\varepsilon$ such that for $p=q>1$, $v_1\leq C$, i.e.,
$$\frac{\nabla u}{u}\leq C,$$
then \eqref{e2} is proved. Assume that $u(x)$ attains its maximum and minimum at $x_1, x_2$, respectively. Then by \eqref{e2}, we have
\begin{eqnarray*}
  \log\frac{\max_{\mathbb{S}^n}u(x)}{\min_{\mathbb{S}^n}u(x)} =\log\frac{u(x_1)}{u(x_2)}&=&\int_0^1\frac{d}{dt}\log\left(u(tx_1+(1-t)x_2)\right)dt \\
  &\leq&|x_1-x_2|\int_0^1\nabla\log u(tx_1+(1-t)x_2)dt\leq C,
\end{eqnarray*}
the proof is completed.
\end{proof}

Let $\overline{u}:=\frac{u}{\min_{\mathbb{S}^n}u}$, then $\overline{u}$ satisfies
$$\sigma_k(\nabla^2\overline{u}+\overline{u}I)=\overline{u}^{p-1+\varepsilon}(\overline{u}^2+|\nabla \overline{u}|^2)^{\frac{k+1-p}{2}}(\min_{\mathbb{S}^n}u)^{\varepsilon}\varphi(x), \quad on~ \mathbb{S}^{n}.$$
By Theorem \ref{c1}, we know that there exist positive constants $C$ and $C'$ depending on $n, k, \min_{\mathbb{S}^n} \varphi$ and $\|\varphi\|_{C^1}$, but independent of $\varepsilon$ such that
\begin{equation}\label{e9}
  1\leq\overline{u}\leq\frac{\max_{\mathbb{S}^n}u}{\min_{\mathbb{S}^n}u}\leq C,
\end{equation}
and
\begin{equation}\label{e10}
  |\nabla\overline{u}|=\frac{u}{\min_{\mathbb{S}^n}u}\frac{|\nabla u|}{u}\leq\frac{\max_{\mathbb{S}^n}u}{\min_{\mathbb{S}^n}u}\frac{|\nabla u|}{u}\leq C'.
\end{equation}
According to \eqref{e9}-\eqref{e10} and Theorem \ref{c2-main}, we have
$$|\nabla^2\overline{u}|\leq C'',$$
where $C''$ is a constant depending on $n, k, \min_{\mathbb{S}^n} \varphi$ and $\|\varphi\|_{C^1}$, but independent of $\varepsilon$.

\subsection{Existence and uniqueness}
We will divide into three steps to complete the proof of Theorem 1.4 for the case $p=q>1$.

\textbf{Step 1:} Existence: Since equation \eqref{e1} has a unique positive spherical convex solution $u_{\varepsilon}$ for any small constant $\varepsilon>0$. Denote $\overline{u}_{\varepsilon}=\frac{u_{\varepsilon}}{\min_{\mathbb{S}^n}u_{\varepsilon}}$, then $\overline{u}_{\varepsilon}$ satisfies
$$\sigma_k(\nabla^2\overline{u}_{\varepsilon}+\overline{u}_{\varepsilon}I)=\overline{u}_{\varepsilon}^{p-1+\varepsilon}(\overline{u}_{\varepsilon}^2+
|\nabla \overline{u}_{\varepsilon}|^2)^{\frac{k+1-p}{2}}(\min_{\mathbb{S}^n}u_{\varepsilon})^{\varepsilon}\varphi(x), \quad on~ \mathbb{S}^{n}.$$
Letting $\varepsilon\rightarrow 0^{+}$, we have $|\nabla(\overline{u}_{\varepsilon})^{\varepsilon}|=\varepsilon(\overline{u}_{\varepsilon})^{\varepsilon-1}|\nabla \overline{u}_{\varepsilon}|\rightarrow0$ by \eqref{e9}-\eqref{e10}. Then $(\min_{\mathbb{S}^n}\overline{u}_{\varepsilon})^{\varepsilon}$ converges to a positive constant $\gamma$. Thus $\overline{u}_{\varepsilon}$ converges to a positive spherical convex solution of equation \eqref{Keq}. Constant rank theorem (Theorem \ref{crt}) can maintain the convexity during the use of the continuity method.

\textbf{Step 2:} Uniqueness: Suppose there are two positive solutions $u$ and $\overline{u}$ such that
$$\sigma_k(u_{ij}+u\delta_{ij})=u^{p-1}(u^2+|\nabla u|^2)^{\frac{k+1-p}{2}}\gamma\varphi(x),$$
and
$$\sigma_k(\overline{u}_{ij}+\overline{u}\delta_{ij})=\overline{u}^{p-1}(\overline{u}^2+|\nabla \overline{u}|^2)^{\frac{k+1-p}{2}}\gamma\varphi(x).$$
Let $M(u):=\frac{\sigma_k(u_{ij}+u\delta_{ij})}{u^{p-1}(u^2+|\nabla u|^2)^{\frac{k+1-p}{2}}}$, then $M(u)-M(\overline{u})=\gamma\varphi-\gamma\varphi=0$. Since $M$ is invariant under scaling, we may assume $u\leq\overline{u}$ and $u(x_0)=\overline{u}(x_0)$ for some point $x_0\in \mathbb{S}^n$. Denote $u_t=tu+(1-t)\overline{u}$ for $0\leq t \leq 1$, then
\begin{eqnarray*}
  0=M(u)-M(\overline{u}) &=& \int_0^1\frac{d}{dt}M(u_t)dt\\
  &=&\sum_{i,j}a_{ij}(x)(u-\overline{u})_{ij}+\sum_ib_i(x)(u-\overline{u})_i+c(x)(u-\overline{u}),
\end{eqnarray*}
where
$$a_{ij}=\int_0^1u_t^{1-p}(u_t^2+|\nabla u_t|^2)^{-\frac{k+1-p}{2}}\frac{\partial\sigma_k}{\partial(W_t)_{ij}}dt>0,$$
$$b_i=-(k+1-p)\int_0^1u_t^{1-p}(u_t^2+|\nabla u_t|^2)^{-\frac{k+1-p}{2}-1}(tu_i+(1-t)\overline{u}_i)\sigma_k(W_t)dt,$$
$$c=\int_0^1u_t^{-p}(u_t^2+|\nabla u_t|^2)^{-\frac{k+1-p}{2}}\left((1-p)\sigma_k(W_t)-\frac{(k+1-p)u_t^2}{u_t^2+|\nabla u_t|^2}\sigma_k(W_t)+u_t\sum_i\frac{\partial\sigma_k}{\partial(W_t)_{ii}}\right)dt,$$
and
$$W_t=t(\nabla^2u+uI)+(1-t)(\nabla^2\overline{u}+\overline{u}I).$$
Therefore by the maximum principle, we have $u-\overline{u}\equiv0$ on $\mathbb{S}^n$.

\textbf{Step 3:} Uniqueness of the constant $\gamma$:
Assume that there exist two positive constants $\gamma, \overline{\gamma}$ and two solutions $u, \overline{u}$ such that
$$\sigma_k(u_{ij}+u\delta_{ij})=u^{p-1}(u^2+|\nabla u|^2)^{\frac{k+1-p}{2}}\gamma\varphi(x),$$
and
$$\sigma_k(\overline{u}_{ij}+\overline{u}\delta_{ij})=\overline{u}^{p-1}(\overline{u}^2+|\nabla \overline{u}|^2)^{\frac{k+1-p}{2}}\overline{\gamma}\varphi(x).$$
Suppose $G=\frac{u}{\overline{u}}$ attains its maximum at $x_0\in\mathbb{S}^n$. Then at $x_0$,
$$0=\nabla\log G=\frac{\nabla u}{u}-\frac{\nabla\overline{u}}{\overline{u}},$$
and
\begin{eqnarray*}
  0 &\geq&\nabla^2\log G \\
  &=&\frac{\nabla^2u}{u}-\frac{(\nabla u)^2}{u^2}-\frac{\nabla^2\overline{u}}{\overline{u}}+\frac{(\nabla \overline{u})^2}{\overline{u}^2}\\
  &=&\frac{\nabla^2u}{u}-\frac{\nabla^2\overline{u}}{\overline{u}},
\end{eqnarray*}
which implies $\sigma_k(u^{-1}(\nabla^2u+uI))\leq\sigma_k(\overline{u}^{-1}(\nabla^2\overline{u}+\overline{u}I))$. Thus at $x_0$ we derive
\begin{eqnarray*}
  \frac{\gamma}{\overline{\gamma}}=\frac{\gamma\varphi(x_0)}{\overline{\gamma}\varphi(x_0)}&=&\frac{(\overline{u}^2+|\nabla\overline{u}|^2)^
  {\frac{k+1-p}{2}}u^{1-p}\sigma_k(\nabla^2u+uI)}{(u^2+|\nabla u|^2)^
  {\frac{k+1-p}{2}}\overline{u}^{1-p}\sigma_k(\nabla^2\overline{u}+\overline{u}I)} \\
  &=&\frac{\sigma_k(u^{-1}(\nabla^2u+uI))}{\sigma_k(\overline{u}^{-1}(\nabla^2\overline{u}+\overline{u}I))}\leq 1.
\end{eqnarray*}
Similarly, we can also get $\gamma\geq\overline{\gamma}$ at the minimum point of $G$, hence $\gamma\equiv\overline{\gamma}$.

\textbf{Conflict of interest statement:}
On behalf of all authors, the corresponding author states that there is no conflict of interest.

\textbf{Data availability statement:}
No datasets were generated or analysed during the current study.


\end{document}